\documentclass{amsart}
\usepackage[utf8]{inputenc}

\usepackage{graphics,thmtools,wasysym}

\usepackage[T2A]{fontenc}
\usepackage[russian,english]{babel}
\usepackage[utf8]{inputenc}

\usepackage{amsthm}
\usepackage{amsbsy,amsmath,amssymb,amscd,amsfonts}
\usepackage[pagebackref=false]{hyperref}
\usepackage{url,booktabs,nicefrac,microtype,comment}

\usepackage{graphicx,float,latexsym,color}
\usepackage[font={scriptsize,it}]{caption}
\usepackage{subcaption}

\usepackage{makecell}

\usepackage[dvipsnames]{xcolor}

\newtheorem{theorem}{Theorem}
\newtheorem{obs}{Observation}
\newtheorem{proposition}{Proposition}
\newtheorem{conjecture}{Conjecture}
\newtheorem{corollary}{Corollary}
\newtheorem{lemma}{Lemma}
\theoremstyle{remark}

\theoremstyle{definition}

\hypersetup{
    pdftoolbar=true,        % show Acrobat’s toolbar?
    pdfmenubar=true,        % show Acrobat’s menu?
    pdffitwindow=false,     % window fit to page when opened
    pdfstartview={FitH},    % fits the width of 
    colorlinks=true,       % false: boxed links; true: colored links
    linkcolor=OliveGreen,          % color of internal links (change box color with linkbordercolor)
    citecolor=blue,        % color of links to bibliography
    filecolor=black,      % color of file links
    urlcolor=red           % color of external links
}

\usepackage{lineno}
\title{Family Ties: Relating Poncelet\\
3-Periodics by their Properties}
\author{Ronaldo Garcia}
\author{Dan Reznik}
\begin{document}

\maketitle

\begin{abstract}
We compare loci types and invariants across Poncelet families interscribed in three distinct concentric Ellipse pairs: (i) ellipse-incircle, (ii) circumcircle-inellipse, and (iii) homothetic. Their metric properties are mostly identical to those of 3 well-studied families: elliptic billiard (confocal pair), Chapple's poristic triangles, and the Brocard porism. We therefore organized them in three related groups.
\end{abstract}

\section{Introduction}
\label{sec:intro}
We have been studying loci and invariants of Poncelet 3-periodics in the confocal ellipse pair (elliptic billiard). Classic invariants include Joachmisthal's constant $J$ (all trajectory segments are tangent to a confocal caustic) and perimeter $L$ \cite{sergei91}.

A few properties detected experimentally \cite{reznik2020-intelligencer} and later proved can be divided into two groups: (i) loci of triangle centers (we use the $X_k$ notation in \cite{etc}), and (ii) invariants. 

In terms of loci, the following results have been proved: (i) the locus of the incenter \cite{garcia2019-incenter,olga14}, barycenter \cite{sergei2016-com}, circumcenter \cite{fierobe2021-circumcenter,garcia2019-incenter}, orthocenter \cite{garcia2020-ellipses} and many others are ellipses; (ii) a special triangle center known as the Mittenpunkt $X_9$ is stationary \cite{olga19_mitten}.

For invariants we chiefly have (i) the sum of cosines \cite{akopyan2020-invariants,bialy2020-invariants}, (ii) the product of outer polygon cosines, and (iii) outer-to-3-periodic area ratio \cite{caliz2020-area-product}.

We continue our inquiry into loci and invariants by now considering 3-periodic families three other non-confocal though concentric ellipse pairs. Referring to  Figure~\ref{fig:four-sys}:

\begin{itemize}
    \item Family I: outer ellipse and incircle, incenter $X_1$ is stationary.
    \item Family II: outer circumcircle and inellipse, circumcenter $X_3$ is stationary
    \item Family III: an axis aligned pair of homothetic ellipses, the barycenter $X_2$ is stationary.
\end{itemize}

\begin{figure}
    \centering
    \includegraphics[width=.8\textwidth]{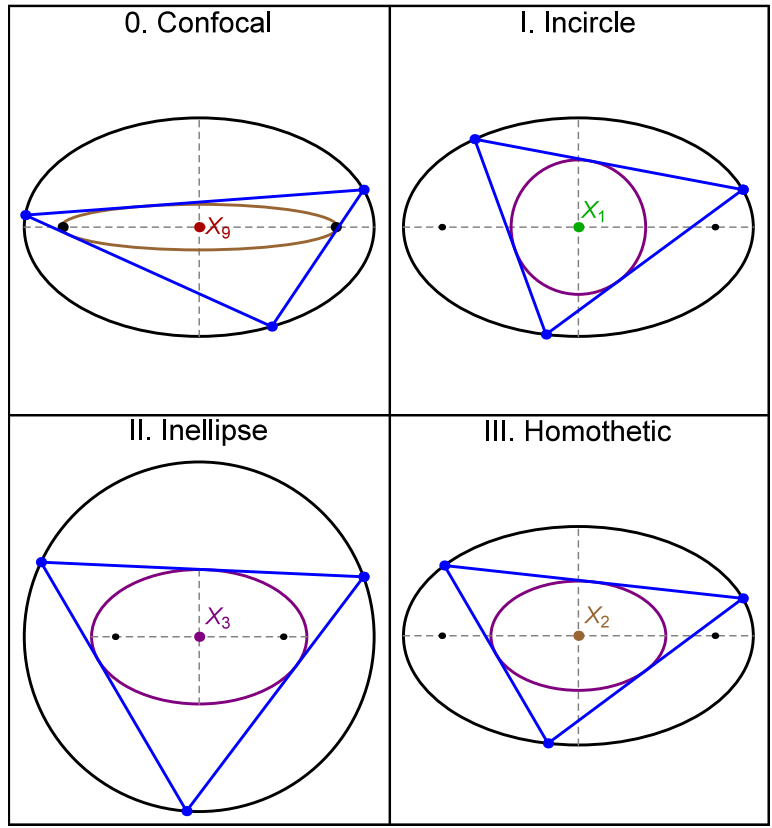}
    \caption{Poncelet 3-periodic families in the various concentric ellipse pairs studied in the article. Properties and loci of the confocal pair (elliptic billiard) were studied in \cite{reznik2020-intelligencer,garcia2020-new-properties,garcia2020-ellipses}. For each family the particular triangle center which is stationary is indicated.}
    \label{fig:four-sys}
\end{figure}

One goal is to identify properties of the above common with previously-studied 3-periodic families, namely, (i) the confocal pair (elliptic billiard), (ii) Chapple's porism \cite{gallatly1914-geometry} and (iii) the so-called Brocard porism \cite{bradley2007-brocard,johnson29}. A quick review of their geometry appears in Section~\ref{sec:review}.

\subsection*{Main Results}

Here are our main results:

\begin{itemize}
    \item Family I
    \begin{itemize}
        \item It conserves the circumradius, the sum of cosines, and the sum of sidelengths divided by their product.
        \item Its sum of cosines is identical to that of the confocal pair which is its affine image.
        \item The family is the image of Chapple's poristic family \cite{odehnal2011-poristic} under a variable rigid rotation.
        \item The poristic family is the image of the confocal family under a variable similarity transform \cite{garcia2020-poristic}. Therefore family I retains several all scale-free invariants identified for the elliptic billiard, including the sum of cosines.
        \end{itemize}
    \item Family II
    \begin{itemize}
           \item It conserves the cosine product and the sum of squared sidelengths.
           \item Its product of cosines is identical to that of the excentral triangles in the confocal pair which is its affine image.
           \item In the elliptic billiard, the locus of the incenter (resp. symmedian point) is an ellipse (resp. quartic) \cite{garcia2020-ellipses}. Here the roles swap: the incenter describes a quartic, and the symmedian is an ellipse.
           \item The orthic triangles of this family are the image of the poristic family under a variable rigid rotation.
    \end{itemize}
    \item Family III
    \begin{itemize}
        \item It conserves area, sum of sidelengths squared, sum of cotangents (the latter implies that the Brocard angle is invariant).
        \item Again in contradistinction with the elliptic billiard, the locus of the incenter $X_1$ is non-elliptic while that of $X_6$ is an ellipse.
        \item The locus of irrational triangle centers $X_k$, $k=$13,14,15,16, i.e., the isodynamic and isogonic points, are circles! In the billiard, they are non-conic.
        \item As shown in \cite{reznik2020-similarityII}, this family is the image of Brocard porism triangles \cite{bradley2007-brocard} under a variable similarity transform.
    \end{itemize}
\end{itemize}

Thus, the following group Poncelet families is proposed with mostly identical properties: (i) family I: confocal, poristics; (ii) family II: confocal excentrals, poristic excentrals; (iii) family III: Brocard porism. Table~\ref{tab:xn-comparison} shows how loci types are shared and/or differ across families, and Figure~\ref{fig:transf-overview} gives a bird's eye view of the kinship across these families via various transformations.

\subsection*{Related Work}

Romaskevich proved the locus of the incenter $X_1$ over the confocal family is an ellipse \cite{olga14}. Schwartz and Tabachnikov showed that the locus of barycenter and area centers of Poncelet trajectories are ellipses though the locus of the perimeter centroid in general isn't a conic \cite{sergei2016-com}. For $N=3$, the former correspond to $X_2$ and the latter to the Spieker center $X_{10}$. Garcia \cite{garcia2019-incenter} and Fierobe \cite{fierobe2021-circumcenter} showed that the locus of the circumcenter of 3-periodics in the elliptic billiard are ellipses. Indeed, the loci of 29 out of the first 100 triangle centers listed in \cite{etc} are ellipses \cite{garcia2020-ellipses}. Tabachnikov and Tsukerman \cite{tabachnikov2015-circum} and Chavez-Caliz \cite{caliz2020-area-product} studied properties and loci of the ``circumcenters of mass'' of Poncelet N-periodics. This is a generalizations of the classical concept of circumcenter to generic polygons, based on triangulations, etc.

The following invariants for N-periodics in the elliptic billiard have been proved: (i) sum of cosines  \cite{akopyan2020-invariants,bialy2020-invariants}, (ii) product of cosines of the outer polygons \cite{akopyan2020-invariants,bialy2020-invariants}, and (iii) area ratios and products of N-periodics and their polar polygons (excentral triangle for N=3); interestingly, these depend on the parity of N \cite{bialy2020-invariants,caliz2020-area-product}. Result (i) also holds for the Poncelet family interscribed between an ellipse and a concentric circle \cite[Corollary 6.4]{akopyan2020-invariants}.

\subsection*{Article structure} We start by reviewing the confocal, Chapple's, and Brocard porisms in Section~\ref{sec:review}. We then describe properties, invariants, and transformations of families I, II, and III in Sections~\ref{sec:I}, \ref{sec:II}, and \ref{sec:III}, respectively. We summarize all results in Section~\ref{sec:summary}. Highlights include (i) a graph representing affine and/or similarity relations between the various families (Figure~\ref{fig:transf-overview}), (ii) a table of conserved quantities which we have found to continue to hold for $N>3$ (proof pending), and (iii) a table with links to videos illustrating some phenomena herein.

\section{Review of Classic Porisms and Proof Method}
\label{sec:review}
Grave's Theorem affirms that given a confocal pair $(\mathcal{E},\mathcal{E}'')$, the two tangents to $\mathcal{E}''$ from a point $P$ on $\mathcal{E}$ will be bisected by the normal of $\mathcal{E}$ at $P$ \cite{miller1925}. A consequence is that any closed Poncelet polygon interscribed in such a pair, if regarded as the path of a moving particle bouncing elastically against the boundary, will be {\em N-periodic}. For this reason, this pair is termed the {\em elliptic billiard}; \cite{sergei91} is the seminal work. It is conjectured as the only integrable planar billiard \cite{kaloshin2018}. One consequence, mentioned above, is that it conserves perimeter $L$. An explicit parametrization for 3-periodic vertices appears in Appendix~\ref{app:explicit-confocal}.

Referring to Figure~\ref{fig:poristic}, poristic triangles are a one-parameter Poncelet family with fixed incircle and circumcircle discovered in 1746 by William Chapple. Recently, Odehnal \cite{odehnal2011-poristic} has studied loci of its triangle centers. showing many of them are either stationary, ellipses, or circles. Surprisingly, the poristic family is the image of billiard 3-periodics under a variable similarity transform \cite{garcia2020-poristic}, and these two families share many properties and invariants.

\begin{figure}
    \centering
    \includegraphics[width=.7\textwidth]{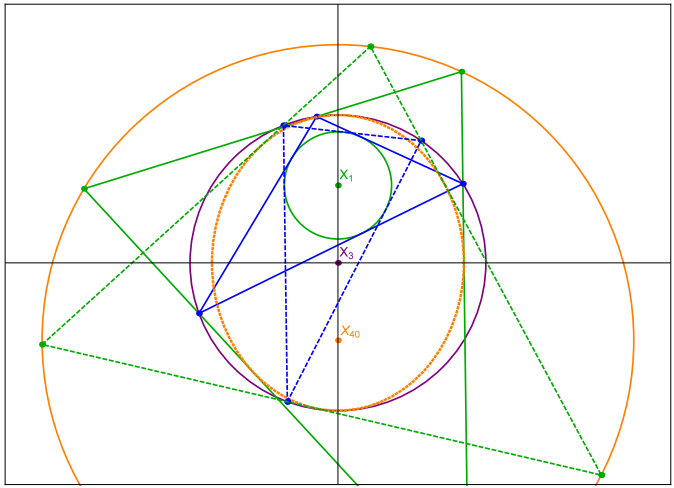}
    \caption{The poristic triangle family (blue) \cite{gallatly1914-geometry} has a fixed incircle (green) and circumcircle (purple). Let $r,R$ denote their radii. Its excentral triangles (green) are inscribed in a circle of radius $2R$ centered on the Bevan point $X_{40}$ and circumscribe the MacBeath inconic (dashed orange) \cite{mw}, centered on $X_3$ with foci at $X_1$ and $X_{40}$. A second configuration is also shown (dashed blue and dashed green).  \href{https://youtu.be/DS4ryndDK6Q}{Video}}
    \label{fig:poristic}
\end{figure}

Referring to Figure~\ref{fig:brocard-porism}, the Brocard porism \cite{bradley2007-brocard} is a family of triangles inscribed in a circle and circumscribed to a special inellipse known as the ``Brocard inellipse'' \cite[Brocard Inellipse]{mw}. Notably, the family's Brocard points are stationary and coincide with the foci of the inellipse. Also remarkable is the fact that the Brocard angle $\omega$ is invariant \cite{johnson29}. In \cite{reznik2020-similarityII} we showed this family is the image of family III triangles under a variable similarity transform.

\begin{figure}
    \centering
    \includegraphics[width=.7\textwidth,trim={00 450 0 0},clip]{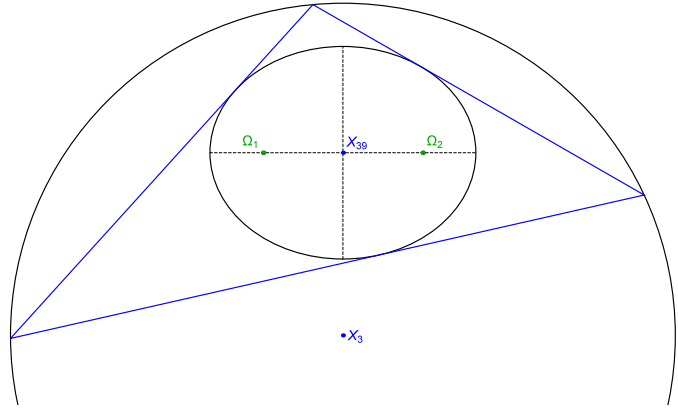}
    \caption{The Brocard porism \cite{bradley2007-brocard} is a 1d Poncelet family of triangles (blue) inscribed in a circle (black, upper half shown) and circumscribed about the Brocard inellipse \cite{mw} centered on $X_{39}$ and with foci at the stationary Brocard points $\Omega_1$ and $\Omega_2$ of the family. The Brocard angle is invariant \cite{johnson29}. \href{https://youtu.be/JANPPLET0so}{Video}}
    \label{fig:brocard-porism}
\end{figure}

\subsection*{A word about our proof method}

We omit some proofs below as they are obtained from a consistent method used previously in \cite{garcia2020-ellipses}: (i) depart from symbolic expressions for the vertices of an isosceles 3-periodic (see Appendix~\ref{app:explicit}); (ii) obtain a symbolic expression for the invariant of interest; (iii) simplify it assisted by a CAS, arriving at a ``candidate'' symbolic expression for the invariant; (iv) verify the latter holds for any (non-isosceles) N-periodic and/or Poncelet pair aspect ratios and if it does, declare it as provably invariant.

\section{Family I: Outer Ellipse, Inner Circle}
\label{sec:I}
Here we study a Poncelet family inscribed in an ellipse centered on $O$ with semi-axes $(a,b)$ and circumscribes a concentric circle of radius $r$, Figure~\ref{fig:ellipse-circle-poristic} (left). An explicit parametrization is provided in Appendix~\ref{app:explicit-I}.

Cayley's closure condition \cite{dragovic11} assumes a simple form for 3-periodics in a concentric, axis-aligned pair of ellipses \cite{georgiev2012-poncelet}:

\begin{proposition}
For 3-periodics in an axis-aligned, concentric ellipse pair:

\begin{equation}
    \frac{a'}{a}+\frac{b'}{b}=1,
    \label{eqn:pair-n3}
\end{equation}
where $a>b>0$, $a'>0$, and $b'>0$.
\end{proposition}

\begin{corollary}
For family I 3-periodics, the radius $r$ of the fixed incircle is given by:
\[ r=\frac{{a}{b}}{a+b}\cdot\]
\end{corollary}

\begin{proposition}
In the family I 3-periodics   the locus of the barycenter $X_2$ is an ellipse with axes $a_2=a(a - b)/(3a + 3b) $ and $b_2=b(a - b)/(3a + 3b)$
 centered on $O=X_1$.

\label{prop:X2_familyI}
\end{proposition}

\begin{theorem}
Family I 3-periodics have invariant circumradius $R=(a+b)/2$. Furthermore, the locus of the circumcenter $X_3$ is a circle of radius $d=R-b=a-R$ centered on $O=X_1$.
\label{prop:R}
\end{theorem}

\begin{proof}
Consider the explicit expressions derived for 3-periodic vertices in Appendix~\ref{app:explicit-I}. Let a first vertex $P_1=(x_1,y_1)$. From this, we obtain the center $X_3$ of the orbit's circumcircle:
{\small
\[X_3=\left[ - \,{\frac {x_1\, \left( a-b \right)  \left( -x_1^{2}
			\left( a+b \right) ^{2}+{a}^{2}b \left( 2\,a+b \right)  \right) }{2a
			\left(  \left( {a}^{2}-{b}^{2} \right) x_1^{2}+{a}^{2}{b}^{2}
			\right) }}
	 ,  {\frac { \left( a-b \right)  \left( x_1^{2} \left( a+b
	 		\right) ^{2}-{a}^{2}{b}^{2} \right) y_1}{2b \left( {a}^{2}x_1^2+{b}^{2} \left( {a}^{2}-x_1^{2} \right)  \right) }}
\right],\]
}
and radius $(a+b)/2$. We also obtain that the locus of $X_3$ is a circle with center $(0,0)$ and radius $(a-b)/2$.

\end{proof}

\begin{proposition}
Over family I 3-periodics the locus of the orthocenter $X_4$ is an ellipse of axes $a_4=(a - b)b/(a + b) $ and $b_4=(a - b)a/(a + b)$
 centered  on $O=X_1$.
\label{prop:X4_familyI}
\end{proposition}

\begin{proposition}
Over family I 3-periodics the locus of the   $X_5$ triangle center is a circle of radius $d=\frac{(a-b)^2}{4(a+b)}$
  centered on $O=X_1$.
\label{prop:X5}
\end{proposition}

\begin{proposition} The power of $O$ with respect to the circumcircle is invariant and  equal to $-a b$.
\end{proposition}

\begin{proof}
Direct, analogous to \cite[Thm.3]{garcia2020-new-properties}.
\end{proof}

\begin{proposition}
Over family I 3-periodics, the locus of $X_6$ is a quartic given by the following implicit equation:

\begin{align*}
    &   \left( b \left( b+2\,a \right)  \left( {a}^{2}+2\,ab+3\,{b}^{2}
 \right)  {x}^{2}+a \left( a+2\,b \right)  \left( 3\,{a}^{2}+2\,ab+{b}
^{2} \right) {y}^{2} \right) ^{2}\\
&-{a}^{2}{b}^{2} \left( a-b \right) ^{
2} \left( {b}^{2} \left( b+2\,a \right) ^{2}{x}^{2}+{a}^{2} \left( a+2
\,b \right) ^{2}{y}^{2} \right) =0
\end{align*}
\end{proposition}
\subsection{Connection with the poristic family}

Below we show that family I 3-periodics is the image of the poristic family \cite{odehnal2011-poristic} under a variable rigid rotation about $X_1$.

Recall the poristic family of triangles with fixed, non-concentric incircle and circumcircle with centers separated by $d=\sqrt{R(R-2r)}$  \cite{gallatly1914-geometry,odehnal2011-poristic}. Let $\mathcal{I}$ be a (moving) reference frame centered on $X_1$ with one axis oriented toward $X_3$. Referring to Figure~\ref{fig:ellipse-circle-poristic} (right):

\begin{theorem}
With respect to $\mathcal{I}$, family I 3-periodics are the poristic triangle family (modulo a rigid rotation about $X_1$).
\end{theorem}

\begin{proof}
This stems from the fact that $R$, $r$, and $d$ are constant. 
\end{proof}

As proved in \cite[Thm.3]{garcia2020-poristic}:

\begin{obs}
The $X_1$-centered circumconic to the poristic family is a rigidly-rotating ellipse with axes $R+d$ and $R-d$.
\end{obs}

Since this circumellipse is identical (up to rotation) to the outer ellipse of family I, then $R+d=a$ which is coherent with Proposition~\ref{prop:R}.

Furthermore, because poristic triangles are the image of billiard 3-periodics under a (varying) affine transform \cite[Thm 4]{garcia2020-poristic}, it displays the same scale-free invariants.

\begin{corollary}\label{cor:soma_produto}
Family I 3-periodics conserve the sum of cosines, product of half-sines, and all scale-free invariants.

\begin{equation}
     \sum_{i=1}^3\cos\theta_i=\frac{a^2 + 4ab + b^2}{(a + b)^2},\;\;\;\prod_{i=1}^3\sin\frac{\theta_i}{2}=\frac{ab}{2(a+b)^2}\cdot
\label{eqn:I-sum-cos}
\end{equation}
\end{corollary}

\begin{figure}
    \centering
    \includegraphics[width=\textwidth]{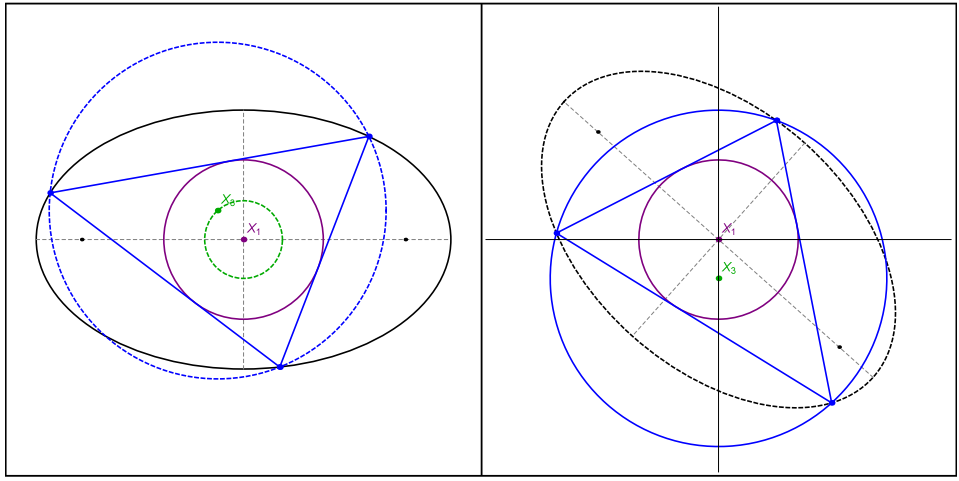}
    \caption{Family I 3-periodics (left) are identical (up to rotation) to the family of poristic triangles (right) \cite{gallatly1914-geometry}, if the former is observed with respect to a reference system where $X_1$ and $X_3$ are fixed. The fixed incircle (resp. circumcircle) are shown purple (resp. blue). The original outer ellipse (black on both drawings) becomes the $X_1$-centered circumellipse in the poristic case. Over the family, this ellipse is known to rigidly rotate about $X_1$ with axes $R+d,R-d$, where $d=|X_3-X_1|$ \cite{garcia2020-poristic}. \href{https://youtu.be/ML_AZoX736w}{Video}}
    \label{fig:ellipse-circle-poristic}
\end{figure}

Note that invariant sum of cosines for family I N-periodics was proved for all $N$ in \cite[Corollary 6.4]{akopyan2020-invariants}. In fact:

\begin{theorem}
Let $(\mathcal{E}_I,\mathcal{E}_I'')$ be a confocal pair of ellipses which is an affine image of a family I pair. Both families have invariant and identical sums of cosines.
\label{thm:famI-confocal-cos}
\end{theorem}

\begin{proof}
Let $\alpha,\beta$ and $\alpha'',\beta''$ denote the semi-axes of $\mathcal{E}_I$ and $\mathcal{E}_I''$, respectively. For the pair to admit a 3-periodic family, the latter are given by \cite{garcia2019-incenter}:

\[\alpha''= \frac{\alpha(\delta-\beta^2)}{{\alpha^2-\beta^2}},\;\;\; \beta''= \frac{\beta(\alpha^2-\delta)}{{\alpha^2-\beta^2}}\cdot\]

Consider the following affine transformation:
\[T(x,y)=\left(\frac{\beta''}{\alpha''}x, y\right). \]

This takes $\mathcal{E}_I$ to an ellipse with semi-axes $(a,b)$, $a= \alpha \frac{\beta''}{\alpha''}$ and $b=\beta$ and the caustic $\mathcal{E}_I''$ to a concentric circle of radius $\beta''$.

In \cite[Thm.1]{garcia2020-new-properties} the following expression was given for invariant $r/R$ in the confocal pair:

\begin{equation}
\frac{r}{R}=\frac{2 (\delta-\beta^2)(\alpha^2-\delta)}{(\alpha^2-\beta^2)^2},\;\;\;\delta = \sqrt{\alpha^4-\alpha^2 \beta^2+\beta^4}\cdot
\label{eqn:rOvR}
\end{equation}

Recall that for any triangle, $\sum_{i=1}^3\cos\theta_i=1+r/R$ \cite[Circumradius, Eqn. 4]{mw}. Plugging $a= \alpha \frac{\beta''}{\alpha''}$ and $b=\beta$ into to \eqref{eqn:I-sum-cos} yields \eqref{eqn:rOvR} plus one.
\end{proof}

It turns out that the proof of \cite[Corollary 6.4]{akopyan2020-invariants} implies that for all N, the cosine sum for family I N-periodics is invariant and identical to the one obtained with its confocal affine image \cite{akopyan2020-invariants}.

A known relation for triangles is that $R\,r=(s_1 s_2 s_3)/(4 s)$, where $s_1,s_2,s_3$ are sidelengths and $s=(s_1+s_2+s_3)/2$ is the semiperimeter. Since both $R$ and $r$ are conserved:

\begin{corollary}
The quantity $(s_1s_2s_3)/(4 s)$ is conserved and is equal to $ab/2.$
\end{corollary}

\section{Family II: Outer Circle, Inner Ellipse}
\label{sec:II}
This family is inscribed in a circle of radius $R$ centered on $O$ and circumscribes a  concentric ellipse with semi-axes $a,b$; see Figure~\ref{fig:II-loci}. An explicit parametrization appears in Appendix~\ref{app:explicit-II}.

For the $N=3$ case, \eqref{eqn:pair-n3} implies $R=a+b$. By definition $X_3$ is stationary at $O$ and $R$ is the (invariant) circumradius. As shown in Figure~\ref{fig:II-loci}:

\begin{proposition}
Over family II 3-periodics, the loci of the orthocenter $X_4$ and nine-point center $X_5$ are concentric circles centered on $X_3=O$, with radii $2 d'$ and $d'$ respectively, where $d'=(a-b)/2$ .
\label{prop:II-loci}
\end{proposition}

\begin{proof}
CAS-assisted algebraic simplification.
\end{proof}

\begin{figure}
    \centering
    \includegraphics[width=.66\textwidth]{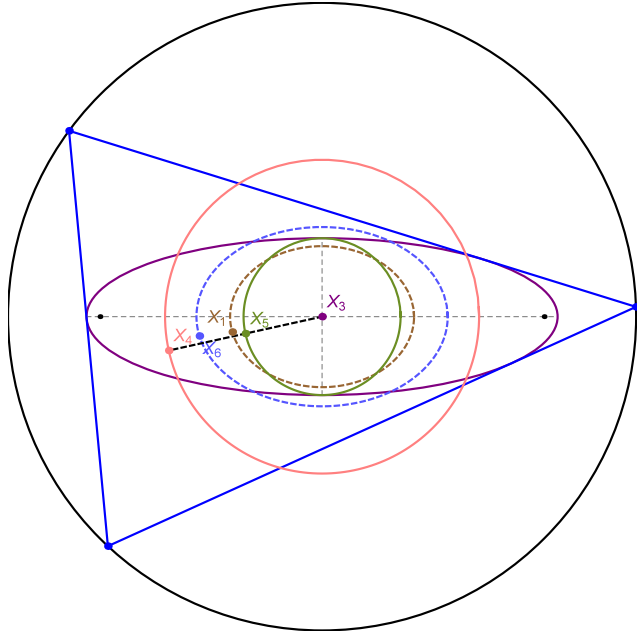}
    \caption{Family II, the $N=3$ case: The loci of both orthocenter $X_4$ (pink) and nine-point center $X_5$ (olive green) are concentric with the external circle (black), with radii $2d'$ and $d'$, respectively. I.e., $|X_4-X_5|=d'$. In contradistinction to the elliptic billiard, the locus of the incenter $X_1$ (dashed brown) is non-elliptic while that of the symmedian point $X_6$ (dashed blue) is an ellipse. \href{https://youtu.be/8xlYaQfQCTw}{Video}.}
    \label{fig:II-loci}
\end{figure}

Recall that in the confocal pair the locus of $X_1$ (resp. $X_6$) is an ellipse (resp. a quartic) \cite{garcia2020-ellipses}; see Appendix~\ref{app:loci-x1x6}. Interestingly:

\begin{proposition}
Over family II 3-periodics, the locus of the symmedian point $X_6$ (resp. the incenter $X_1$) is an ellipse (resp. the convex component of a quartic -- note the other component corresponds to the locus of the 3 excenters which can be concave). These are given by:

\begin{align*}
\text{locus of $X_6$}: &\; \frac{x^2}{a_6^2}+\frac{y^2}{b_6^2}=1, \; a_6=\frac{ a^2 - b^2}{a + 2 b},\;\;  b_6=\frac{ a^2 - b^2}{2a +  b},\\
\text{locus of $X_1$}:&\; \left( {x}^{2}+{y}^{2} \right) ^{2}-2\, \left( a+3\,b \right) 
 \left( a+b \right) {x}^{2}-2\, \left( a+b \right)  \left( 3\,a+b
 \right) {y}^{2} \\
 &+\left( a^2-b^2 \right) ^{2}=0\cdot
 \label{thm:II-loci_X1}
\end{align*}
\end{proposition}

\begin{proof}
CAS-assisted simplification.
\end{proof}

\noindent Let $s_i$ denote the sidelengths of an $N$-periodic.

\begin{theorem}
Family II 3-periodics conserve $L_2=\sum_{i=1}^3 s_i^2=4(a + 2b)(2a + b)$.
\end{theorem}

\begin{proof}
Direct, using the parametrization for vertices in Appendix \ref{app:explicit-II}.
\end{proof}

Note: the above is true for all $N$ \cite[Thm.8, corollary]{akopyan2020-invariants}.

\subsection{Family II and the poristic family}

Below we show that the orthic triangles of Family II 3-periodics are the image of the poristic family \cite{odehnal2011-poristic} under a variable rigid rotation about $X_3$.

\begin{lemma}
Family II 3-periodics conserve the product of cosines, given by:
\[ \prod_{i=1}^3{\cos\theta_i}=\frac{a b}{2 (a+b)^2}\cdot\]
\label{lem:cos-prod-II}
\end{lemma}

\begin{proof}
CAS-assisted simplification.
\end{proof}

The orthic triangle has vertices at the feet of a triangle's altitudes \cite{mw}. Let $R_h$ denote its circumradius. A known property is that $R_h=R/2$ \cite[Orthic Triangle, Eqn. 7]{mw}. Therefore, it is invariant over family II 3-periodics. Referring to Figure~\ref{fig:II-poristic} (left):

\begin{proposition}
The inradius $r_h$ of family II orthic triangles is invariant and given by $r_h=a b/(a+b)$. 
\label{prop:II-orthic-radii}
\end{proposition}

\begin{proof}
$r_h=2 R \prod_{i=1}^3{\cos\theta_i}$ \cite[Orthic Triangle, Eqn. 5]{mw}. Referring to Lemma~\ref{lem:cos-prod-II} completes the proof.
\end{proof}

\begin{figure}
    \centering
    \includegraphics[width=\textwidth]{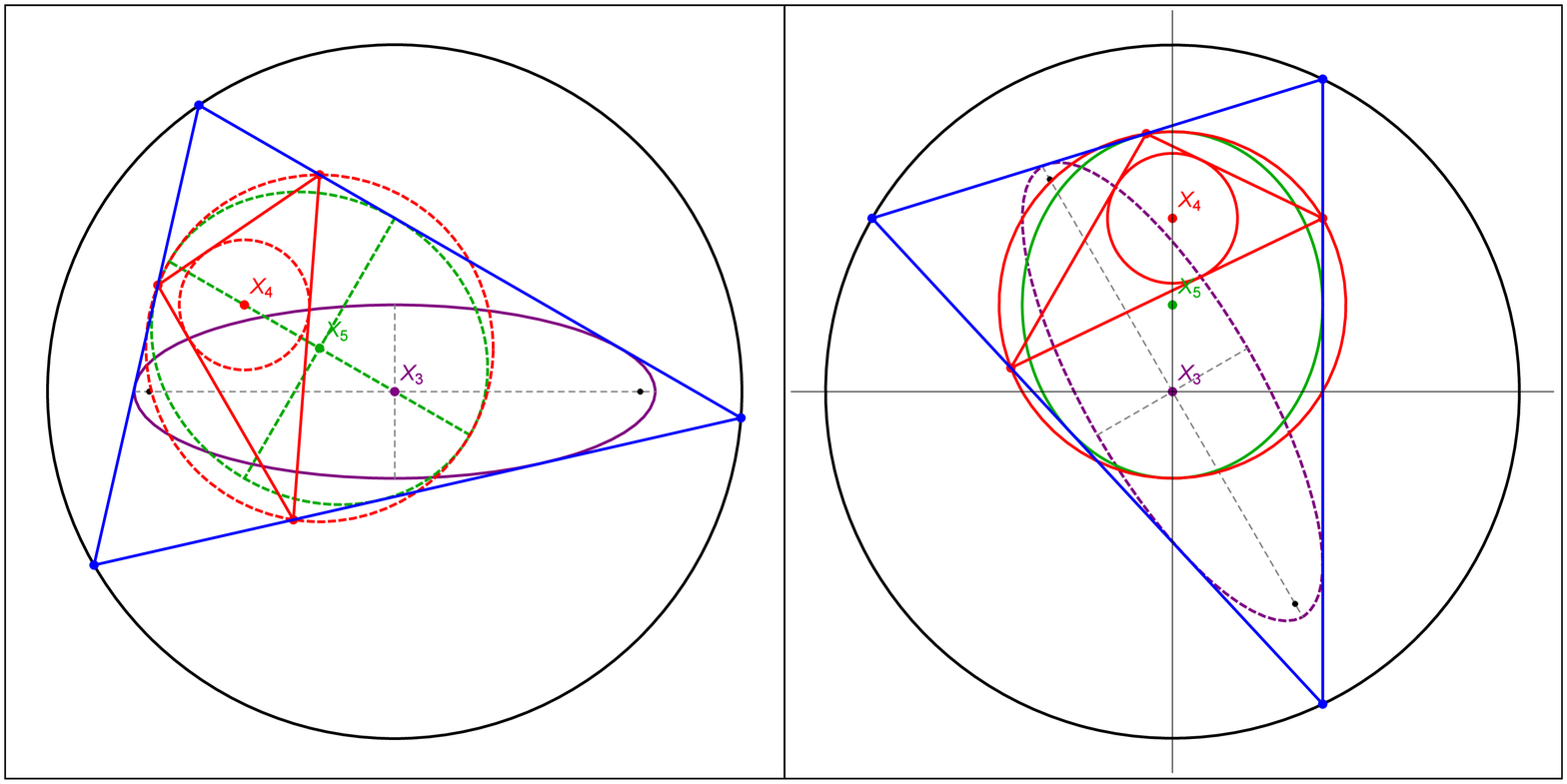}
    \caption{\textbf{Left:} Family II 3-periodics (blue), and their orthic triangle (red). The latter's inradius and circumradius are invariant. The orthic triangle's incircle and circumcircle (both dashed red) are centered on the 3-periodic's orthocenter $X_4$ and the nine-point center $X_5$, respectively. Also shown is the rigidly-rotating MacBeath inellipse (dashed green), centered on $X_5$ with foci at $X_3$ and $X_4$. \textbf{Right:} Family II orthic triangles are identical (up to a variable rotation), to the poristic triangles (red) \cite{odehnal2011-poristic}. Equivalently, the original family is that of poristic excentral triangles (blue), for which both incircle and circumcircle (solid red) are stationary. Also stationary is the excentral MacBeath inellipse (green), i.e., it is the caustic \cite{garcia2020-poristic}, with center $X_5$ and foci $X_3$, and $X_4$, respectively. The original outer circle (black on both images) is also stationary on the poristic case, however the inner ellipse in the Poncelet pair (purple) becomes a rigidly-rotating $X_3$-centered excentral inellipse (dashed purple), whose axes are $R+d'$ and $R-d'$.
    \href{https://youtu.be/wUu2iMesv3U}{Video 1}, \href{https://youtu.be/xM1SAZO9bDc}{Video 2}}
    \label{fig:II-poristic}
\end{figure}

Let $(\mathcal{E}_{II},\mathcal{E}''_{II})$ denote the confocal pair which is an affine image of a circle-inellipse concentric pair. Let $\alpha,\beta$ and $\alpha'',\beta''$ denote the semi-axes of  $\mathcal{E}_{II}$, and $\mathcal{E}''_{II}$, respectively.

%Explicitly $T(x,y)=(x,\frac{a}{b}y)$
\begin{theorem}
The invariant product of cosines for family II triangles is identical to the one obtained from excentral triangles of 3-periodics in $(\mathcal{E}_{II},\mathcal{E''}_{II})$.
\label{thm:famII-confocal-cos-prod}
\end{theorem}

\begin{proof}
Excentrals in the confocal pair conserve the product of cosines \cite[Corollary 2]{garcia2020-new-properties}. Recall that for any triangle:

\[ \prod_{i=1}^{3}{|\cos\theta_i'|}=\frac{r}{4R}\cdot \]

\noindent where $\theta_i'$ are the angles of the excentral triangle. 
%Plug constant expressions for $r_h=a b/(a+b)$ and $R_h=R/2=(a+b)/2$ into the above and verify it yields Lemma~\ref{lem:cos-prod-II}.
Plugging $a=  {\alpha''}$ and $b=\frac{\alpha}{\beta}\beta''$  into \eqref{lem:cos-prod-II} yields four times the above identity when $r/R$ is computed as in \eqref{eqn:rOvR}, completing the proof.
\end{proof}

\begin{lemma}
Family II 3-periodics are always acute.
\label{lem:acute}
\end{lemma}

\begin{proof}
Since $X_3$ is the common center and is internal to the caustic, it will be interior to Family II 3-periodics, i.e., the latter are acute.
\end{proof}

Let $\mathcal{I}'$ be a (moving) reference frame centered on $X_3$ with one axis oriented toward $X_5$ (or $X_4$ as these 3 are collinear). Referring to Figure~\ref{fig:ellipse-circle-poristic} (right):

\begin{theorem}
With respect to $\mathcal{I}'$, family II 3-periodics are the excentral triangles to the poristic family (modulo a rigid rotation about $X_3$). Equivalently, family II orthics are identical (up to said variable rotation) to the poristic triangles.
\end{theorem}

\begin{proof}
$X_5$ of a reference triangle is $X_3$ of the orthic triangle \cite{etc}. Since the family is always acute (Lemma~\ref{lem:acute}), $X_4$ of the reference is $X_1$ of the orthic triangle \cite{coxeter67}. By  Proposition~\ref{prop:II-loci}, $d'=|X_5-X_3|$ is invariant, i.e., the distance between $X_1$ and $X_3$ of the orthic triangle is invariant. The claim follows from noting $X_3,X_5,X_4$ are collinear \cite{mw} and that the orthic inradius and
circumradius are invariant, Proposition~\ref{prop:II-orthic-radii}. 
\end{proof}

Recall from \cite[Thm.2]{garcia2020-poristic}:

\begin{obs}
The $X_3$-centered inconic to the poristic excentral triangles is a rigidly-rotating ellipse with axes $R+d'$ and $R-d'$.
%\textcolor{red}{ Seria 2a e 2b ($R=a+b. d'=(a-b)/2$) 
%ou $(3a+b)/2 e (a+3b)/2$}
\end{obs}

Which makes sense when one considers the rotating reference frame. Also recall from \cite[Thm.1]{garcia2020-poristic} that:

\begin{obs}
The MacBeath Inconic to the excentrals is stationary with axes $R$ and $\sqrt{R^2-d'^2}$.     
%\textcolor{red}{ Seria  $\sqrt{(a+3b)(3a+b)}{2}$  ($R=a+b. d'=(a-b)/2$)}
\end{obs}

Therefore its focal length is simply $2d'=|X_4-X_3|$. Furthermore, because poristic triangles are the image of billiard 3-periodics under a (varying) affine transform \cite[Thm.4]{garcia2020-poristic}, Family II 3-periodics will share all scale-free invariants with billiard excentrals, such as product of cosines, ratio of area to its orthic triangle, etc., see \cite{reznik2020-invariants}.

\section{Family III: Homothetic}
\label{sec:III}
This family is inscribed in an ellipse centered on $O$ with semi-axes $(a,b)$ and circumscribes an homothetic, axis-aligned, concentric ellipse with semi-axes $(a'',b'')$; see Figure~\ref{fig:brocard-n3}. An explicit parametrization is provided in Appendix~\ref{app:explicit-III}.

\begin{proposition}
For family III 3-periodics, $a''=a/2$ and $b''=b/2$, the barycenter $X_2$ is stationary at $O$ and the area $A$ is invariant and given by:
\[A= \frac{3\sqrt{3}}{4} a b \cdot\]
\end{proposition}

\begin{proof}
family III is the affine image of a family of equilateral triangles interscribed within two concentric circles. The inradius of such a family is half its circumradius. Amongst triangle centers, the barycenter $X_2$ is uniquely invariant under affine transformations; it lies at the origin for an equilateral. Affine transformations preserve area ratios. $A$ is the area of an equilateral triangle inscribed in a unit circle scaled by the Jacobian $a b$. This completes the proof.
\end{proof}

\begin{figure}
    \centering
\includegraphics[width=\textwidth]{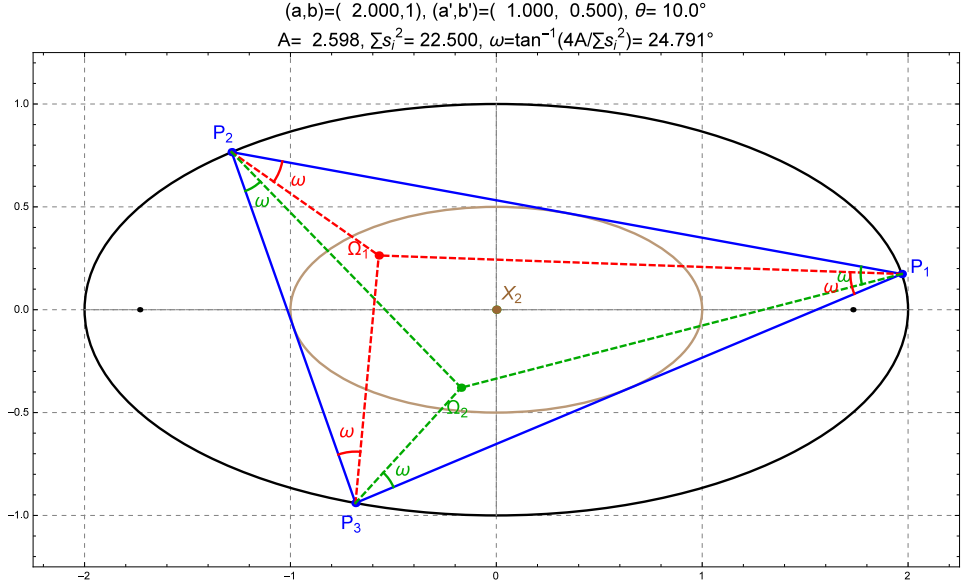}
    \caption{Family III (homothetic pair) 3-periodics (blue). Also shown are the Brocard points $\Omega_1$ and $\Omega_2$. Since both area and sum of squared sidelengths are constant, so is the Brocard angle $\omega$. \href{https://youtu.be/2fvGd8wioZY}{Video}}
    \label{fig:brocard-n3}
\end{figure}

A known result is that the cotangent of the Brocard angle $\cot\omega$ of a triangle is equal to the sum of the cotangents of its three internal angles \cite[Brocard Angle, Eqn. 1]{mw}. Surprisingly, we have:

\begin{proposition}
Family III 3-periodics have invariant $\omega$ given by:
\label{prop:somacotan}

\[ 	\cot\omega=\sum_{i=1}^3 \cot\theta_i= \frac{\sqrt{3}}{2} \frac{{a}^{2}+{b}^{2}}{a b}\cdot
\]	
\end{proposition}

\begin{proof} Direct calculations using the explicit parametrization of vertices in Appendix \ref{app:explicit-III}.
\end{proof}

\noindent A known relation is $\cot\omega=(\sum_{i=1}^3 s_i^2)/(4 A)$ \cite[Brocard Angle, Eqn. 2]{mw}. Therefore, we have:

\begin{corollary}
\label{prop:somaL2}
The sum of squared sidelengths $s_i^2$ is invariant and given by:
	
\[ \sum_{i=1}^3 s_i^2=\frac{9}{2} \left({a}^{2}+{b}^{2}\right)\cdot
\]
\end{corollary}

%\begin{corollary}
%\[ 4\left(1-\cos\left(\frac{2\pi}{n}\right) \right)A_n Cot_n+ n\cos\left(\frac{2\pi}{n}\right) L_n=0\]
%\end{corollary}

As mentioned above, in the confocal pair the loci of $X_1$ (resp. $X_6$) is an ellipse (resp. a quartic) \cite{garcia2020-ellipses}; see Appendix~\ref{app:loci-x1x6}. Interestingly, we have:

\begin{proposition}
For family III, the locus of the incenter $X_1$ (resp. symmedian point $X_6)$ is a quartic (resp. an ellipse). These are given by:

\begin{align*}
\text{locus of $X_1$}: \;& 
16\, \left( {a}^{2}{y}^{2}+{b}^{2}{x}^{2} \right)  \left( {a}^{2}{x}^{
2}+{b}^{2}{y}^{2} \right) -8\,{b}^{2} \left( {a}^{4}+5\,{a}^{2}{b}^{2}
+2\,{b}^{4} \right) {x}^{2}\\
&-8\,{a}^{2} \left( 2\,{a}^{4}+5\,{a}^{2}{b}
^{2}+{b}^{4} \right) {y}^{2}+{a}^{2}{b}^{2} \left( a^2-b^2 \right) ^{2}=0,
 \\
\text{locus of $X_6$}: &\; \frac{x^2}{a_6^2}+\frac{y^2}{b_6^2}=1,\;a_6=\frac{a(a^2-b^2)}{2(a^2+b^2)},\;\;\; b_6=\frac{b(a^2-b^2)}{2(a^2+b^2)} \cdot\\
\label{thm:III-loci_X1}
\end{align*}
\end{proposition}

\begin{proof}
CAS-assisted simplification.
\end{proof}

\subsection{Surprising Circular Loci}

The two isodynamic points $X_{13}$ and $X_{14}$ as well as the two isogonic points $X_{15}$ and $X_{16}$ have trilinear coordinates which are irrational on the sidelengths of a triangle \cite{etc}. In the elliptic billiard their loci are non-elliptic. Indeed, in the elliptic billiard we haven't yet found any triangle centers with a conic locus whose trilinears are irrational. Referring to Figure~\ref{fig:circ-loci-III}, for family III, this is a surprising fact:

\begin{proposition}
The loci of of $X_k$, $k=$13,14,15,16 are circles. Their radii are $(a-b)/2$, $(a+b)/2$, $(a-b)^2/z$, and $(a+b)^2/z$, respectively, where $z=2(a+b)$.
\end{proposition}

\begin{obs}
Over all $a/b$, the radius of $X_{16}$ is minimum when $a/b=3$.
\end{obs}

\begin{figure}
    \centering
    \includegraphics[width=\textwidth]{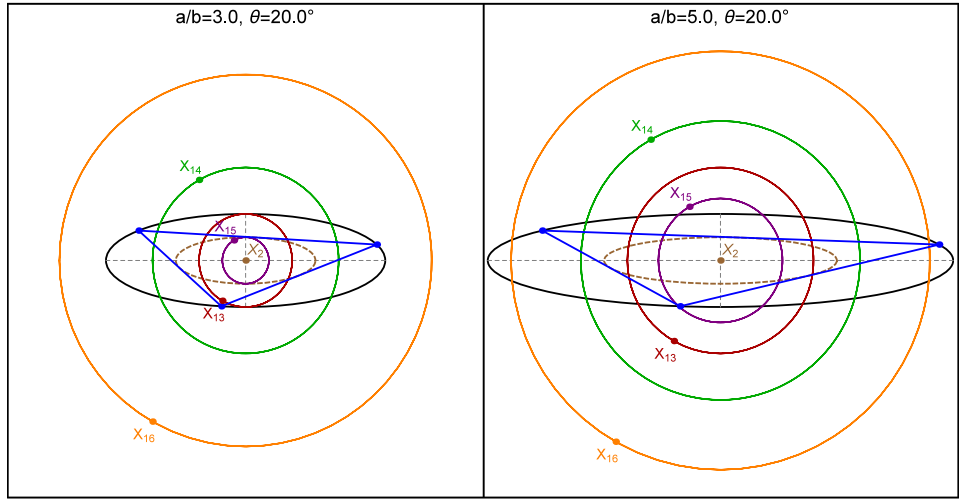}
    \caption{Circular loci of the first and second Fermat points $X_{13}$ and $X_{14}$ (red and green) as well as the first and second isodynamic points $X_{15}$ and $X_{16}$ (purple and orange) for two aspect ratios of the homothetic pair: $a/b=3$ (left) and $a/b=5$ (right). The radius of the $X_{16}$ locus is minimal at the first case. \href{https://youtu.be/ZwTfwaJJitE}{Video}}
    \label{fig:circ-loci-III}
\end{figure}

\subsection{Family III and the Brocard Porism}

The Brocard porism \cite{bradley2007-brocard} is a family of triangles inscribed in a circle and circumscribed about a special inellipse known as the ``Brocard inellipse'' \cite[Brocard Inellipse]{mw}. Its foci coincide with the stationary Brocard points of the family. Furthermore, this family conserves the Brocard angle $\omega$. 

Referring to Figure~\ref{fig:brocard-n3}, we showed that over the homothetic family, the aspect ratio of the Brocard inellipse is invariant \cite{reznik2020-similarityII}. This leads to the following result, reproduced from  \cite[Theorem 3]{reznik2020-similarityII}:

\begin{theorem}
The 3-periodic family in a homothetic pair and that of the Brocard porisms are images of one another under a variable similarity transform.
\end{theorem}

As shown in \cite{garcia2021-ellipses-web}, the locus of the center $X_{39}$ of the Brocard inellipse is an ellipse (it is stationary in the Brocard porism).

\begin{figure}
    \centering
    \includegraphics[width=\textwidth,trim={0 0 0 2.33cm},clip]{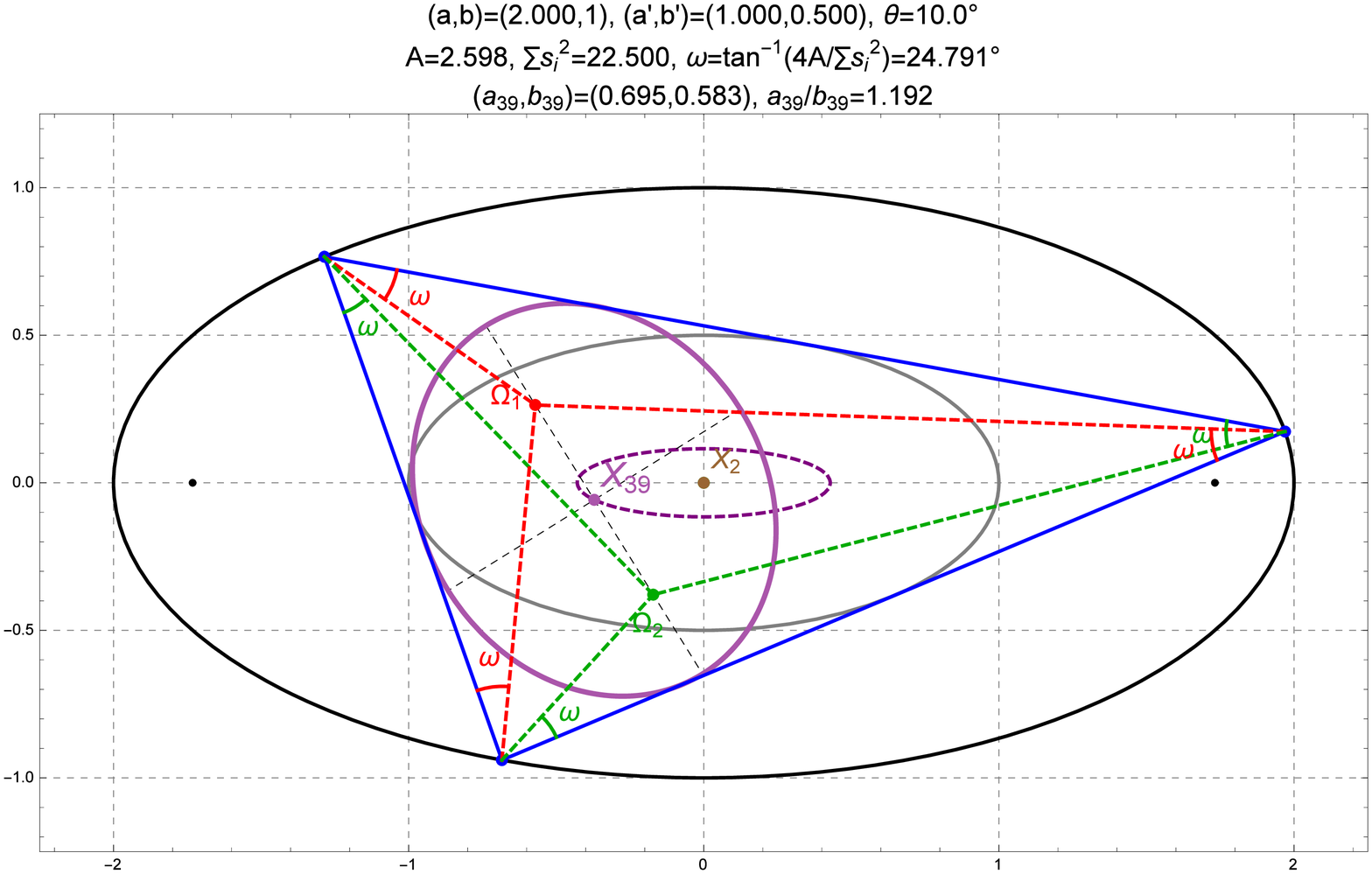}
    \caption{Family III triangles (blue) are the image of Brocard porism triangles under a variable similarity transform \cite{reznik2020-similarityII}. This stems from the fact that the family's Brocard inellipse (purple), centered on $X_{39}$ and with foci on the Brocard points $\Omega_1,\Omega_2$, has a fixed aspect ratio. Also shown is the elliptic locus of $X_{39}$. \href{https://youtu.be/DIm2qTxGWXE}{Video}}
    \label{fig:homot-broc-inell}
\end{figure}

\section{Summary}
\label{sec:summary}
Table~\ref{tab:xn-comparison} summarizes the types of loci (point, circle, ellipse, etc.) for several triangle centers for all families mentioned above. These are organized within three groups A, B, and C with closely-related loci types. Exceptions are also indicated though we still lack a theory for it.

\begin{table}
\small
\centering
\begin{tabular}{|c||c|c|c||c|c|c||c|c||}
\hline
& \multicolumn{3}{c||}{Group A} & \multicolumn{3}{c||}{Group B} & \multicolumn{2}{c||}{Group C} \\
\hline
 & Conf. & F.I & Por. & \makecell{Conf.\\Exc} & F.II &  \makecell{Por.\\Exc.} & F.III & Broc. \\
 \hline
$X_1$ & E & P & P & X&X & X & 4 & X \\
$X_2$ & E & E & C & E&C & P & P & C \\
$X_3$ & E & C & P & E&P & P & E & P \\
$X_4$ & E & E & C & E&C & P & E & C \\
$X_5$ & E & C & C & E&C & P & E & C \\
$X_6$ & 4 & 4 & \textcolor{red}{\textbf{E}} & P&E & C & E & P \\
$X_7$ & E & E & C & X & X & X & X & X \\
$X_8$ & E & E & C & X & X & X & X & X \\
$X_9$ & P & E & C & X & X & X & X & X \\
$X_{10}$ & E & E & C & X & X & X & X & X \\
$X_{11}$ & E$''$ & C$''$ & C$''$ & X & X & \textcolor{red}{\textbf{C$_5$}} & X & X \\
$X_{12}$ & E & C & C & X & X & X & X & X \\
$X_{13}$ & X & X & X & X & X & X & C & C \\
$X_{14}$ & X & X & X & X & X & X & C & C \\
$X_{15}$ & X & X & X & X & X & X & C & P \\
$X_{16}$ & X & X & X & X & X & X & C & P \\
$X_{99}$ & X & X & C$'$ & \textcolor{red}{\textbf{X}}&C$'$ & C$'$ & E$'$ & C$'$ \\
$X_{100}$ & E$'$ & E$'$ & C$'$ & \textcolor{red}{\textbf{X}}&C$'$ & C$'$ & \textcolor{red}{\textbf{X}} & C$'$ \\
$X_{110}$ & X & X & C$'$ & E$'$&C$'$ & C$'$ & \textbf{X} & C$'$ \\
 \hline
\end{tabular}
\caption{Types of loci for several triangle centers over several Poncelet triangle families, divided in 3 groups A,B,C with closely-related metric phenomena: (i) confocal, fam. I, poristics; (ii) confocal excentral, fam. II, poristic excentral triangles; (iii) fam. III and Brocard porism. Symbols P, C, E, and X indicate point, circle, ellipse, and non-elliptic (degree not yet derived) loci, respectively. A number refers to the degree of the non-elliptic implicit, e.g., '4' for quartic. A singly (resp. doubly) primed letter indicates a perfect match with the outer (resp. inner) conic in the pair. The symbol C$_5$ refers to the nine-point circle. The boldface entries indicate a discrepancy in the group (see text). Note: $X_n$ for the confocal and poristic excentral triangles refer to triangle centers of the family itself (not of their reference triangles).}
\label{tab:xn-comparison}
\end{table}

The first row reveals that out of the 8 families considered only in the confocal case is the locus of the incenter $X_1$ an ellipse. Additionaly experimentation has suggested an intriguing conjecture:

\begin{conjecture}
Given a pair of conics which admits a Poncelet 3-periodic family, only when such conics are confocal will the locus of either the incenter $X_1$ or the excenters be a non-degenerate conic.
\end{conjecture}

The plethora of circles in the poristic family had already been shown in \cite{odehnal2011-poristic}. An above-than-expected frequency of ellipses for the confocal pair was signalled in \cite{garcia2020-ellipses}. As mentioned above, irrational centers $X_k$, $k\in[13,16]$ sweep out circles for the homothetic pair. $X_{15}$ and $X_{16}$ are known to be stationary over the Brocard family \cite{bradley2007-brocard}, however the locus of $X_{13}$ and $X_{14}$ are circles! Also noticeable is the fact that (i) though in the confocal pair the locus of $X_1$ and $X_6$ is an ellipse and a quartic, respectively, in both family II and family III said locus types are swapped. The reasons remain mysterious.

It is well-known that there is a projective transformation that takes any Poncelet family to the confocal pair,  \cite{dragovic11}. In this case only  projective properties are preserved. 
If one restricts the set of possible transformations to either affine ones or similarities (which include rigid transformations), one can construct the two-clique graph of interrelations shown in Figure~\ref{fig:transf-overview}.

As mentioned above, the confocal family is the affine image of either family I or family II. In the first (resp. second) case the caustic (resp. outer ellipse) is sent to a circle. Though the affine group is non-conformal, we showed above that both families conserve their sum of cosines (Theorem~\ref{thm:famI-confocal-cos}). One way to see this is that there is an alternate, conformal path which takes family I triangles to the confocal ones, namely a rigid rotation (yielding poristic triangles), followed by a variable similarity (yielding the confocal family).

A similar argument is valid for family II triangles: there is an affine path (non-conformal) to the confocal family though both conserve the product of cosines (Theorem~\ref{thm:famII-confocal-cos-prod}). Notice an alternate conformal composition of rotation (yielding poristic excentral triangles) and a variable similarity (yielding confocal excentral triangles). All in this path conserve the product of cosines.

Finally, family III and Brocard porism triangles form an isolated clique. As mentioned in \cite{reznik2020-similarityII}, these are variable similarity images of one another but cannot be mappable to the other families via similarities nor affinely.

Table~\ref{tab:summary} summarizes some properties of 3-periodics mentioned herein. The last column reveals that many of the invariants continue to hold for N>3. Animations illustrating some focus-inversive phenomena are listed in Table~\ref{tab:playlist}.

\begin{figure}
 \centering
 \includegraphics[width=\textwidth,trim={50 0 130 0},clip]{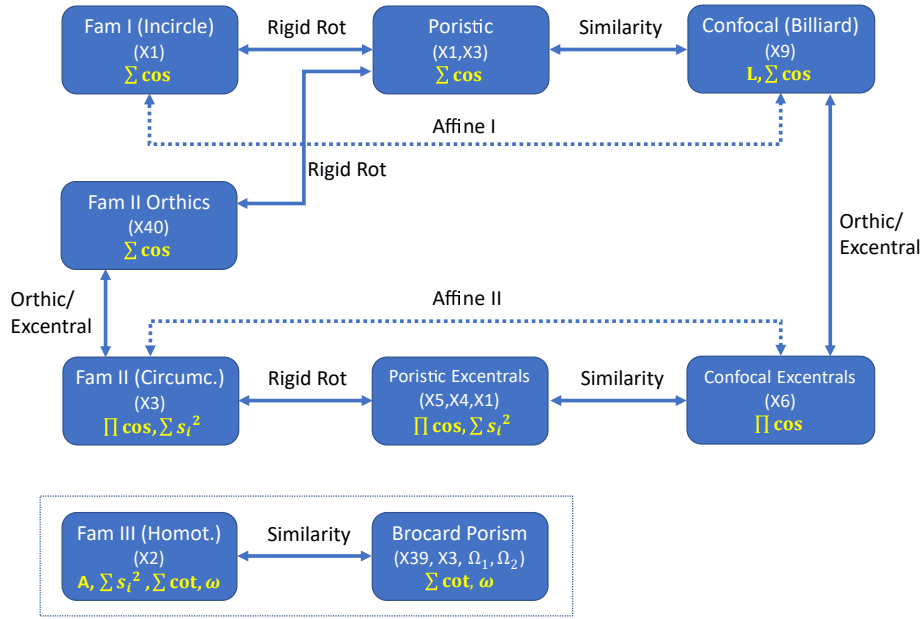}
 \caption{Diagram of transformations that take one 3-periodic family into another. The families are specified in each box while the transformations label the arrows. The second (resp. third) line in each box lists the stationary point(s) (resp. main invariants) in the family.}
 \label{fig:transf-overview}
\end{figure}

\begin{table}
\begin{tabular}{|l|l|l|l|l|l|}
\hline
fam. & pair & \makecell[lt]{N=3\\outer conic} & \makecell[lt]
{N=3\\inner conic} & \makecell[lt]{N=3\\invariants} & \makecell[lt]{N>3} \\
\hline
 & billiard & ellipse $(a,b)$ & confocal caustic & $L,J,r/R,\color{red}\sum\cos$ & $L,J,\color{red}\sum\cos$ \\
I & inner circle & ellipse $(a,b)$ & circle $r=\frac{{a}{b}}{a+b}$ & $R,r/R,\color{red}\sum\cos$ & $\color{red}\sum\cos$ \\
II & outer circle & $R=(a+b)$ & ellipse $(a,b)$ & $\sum{s_i^2},\color{red}\prod\cos$ & $\sum{s_i^2},\color{red}\prod\cos$ \\
III & homothetic & ellipse $(a,b)$ & ellipse $(a/2,b/2)$ & $A,\sum{s_i^2},\omega,\color{red}\sum\cot$ & $A,\sum{s_i^2},\color{red}\sum\cot$ \\
%4 & $X_4$ stationary & ellipse $(a,b)$ & ellipse $(a_0,b_0)$ & $X_4=0$ & pseudo-$X_4$ = 0 \\
\hline
\end{tabular}
\caption{Summary of properties across different concentric Poncelet families. The last column shows some invariants which continue to hold for N>3.}
\label{tab:summary}
\end{table}

\begin{table}
\small
\begin{tabular}{|c|c|c|l|l|}
\hline
id & family & N & Title & \textbf{youtu.be/...}\\
\hline
01 & all & 3 & {Concentric Poncelet families} &
\href{https://youtu.be/8hkeksAsx0E}{\texttt{8hkeksAsx0E}}\\
02 & por. & 3 & {Chapple's poristic family \& excentral triangles} &
\href{https://youtu.be/DS4ryndDK6Qo}{\texttt{DS4ryndDK6Qo}}\\
03 & por. & 3 & {Poristics are image of billiard 3-periodics} &
\href{https://youtu.be/NvjrX6XKSFw}{\texttt{NvjrX6XKSFw}}\\
\hline
04 & I & 3 & \makecell[lt]{Side-by-side w/ the poristic family} &
\href{https://youtu.be/ML_AZoX736w}{\texttt{ML\_AZoX736w}}\\
05 & I & 3,5 & {Circular loci of X3 \& Steiner's curvature centroid} &
\href{https://youtu.be/601OfxuSDGc}{\texttt{601OfxuSDGc}}\\
06 & I & 3,5 & {Invariant ratio of sidelength product to sum} &
\href{https://youtu.be/7Jg2nRkkUhQ}{\texttt{7Jg2nRkkUhQ}}\\
\hline
07 & II & 3 & {Family is image of poristic excentrals} &
\href{https://youtu.be/wUu2iMesv3U}{\texttt{wUu2iMesv3U}} \\
08 & II & 3 & \makecell[lt]{Side-by-side w/the poristic family} &
\href{https://youtu.be/xM1SAZO9bDc}{\texttt{xM1SAZO9bDc}}\\
09 & II & 3,5 & {Circular locus of generalized orthocenter} &
\href{https://youtu.be/3f6YBohQCFg}{\texttt{3f6YBohQCFg}} \\
\hline
10 & III & 3 & {Stationary $X_2$ and invariant Brocard angle} &
\href{https://youtu.be/2fvGd8wioZY}{\texttt{2fvGd8wioZY}}\\
11 & III & 3 & {Loci of $X_k$, $k=$13,14,15,16 are all circles!} &
\href{https://youtu.be/ZwTfwaJJitE}{\texttt{ZwTfwaJJitE}}\\
12 & III & 3 & {Family is image of Brocard porism} &
\href{https://youtu.be/h3GZz7pcJp0}{\texttt{h3GZz7pcJp0}} \\
\hline
13 & I,II & 5,6 & {Locus of generalized circum- and orthocenter} &
\href{https://youtu.be/ZfQEDujbirQ}{\texttt{ZfQEDujbirQ}} \\
14 & I,II & 5 & {Locus of generalized circumcenter} &
\href{https://youtu.be/RP18B827l5I}{\texttt{RP18B827l5I}} \\
15 & I,II & 5 & {Generalized circumcenter (Steiner's curv. centroid)} &
\href{https://youtu.be/RP18B827l5I}{\texttt{RP18B827l5I}}\\
\hline
16 & dual & 3 & {The dual pair: stationary orthocenter} &
\href{https://youtu.be/fpd_Zot5cKk}{\texttt{fpd\_Zot5cKk}}\\
17 & dual & 3--8 & {Generalized stationary orthocenter} & \href{https://youtu.be/ttKjzWeG5B8}{\texttt{ttKjzWeG5B8}}\\
18 & dual & 5,7 & {Generalized stationary orthocenter} & \href{https://youtu.be/gNHiZvBhKF8}{\texttt{gNHiZvBhKF8}}\\
\hline
\end{tabular}
\caption{Videos illustrating some phenomena mentioned herein. The last column is clickable and provides the YouTube code.}
\label{tab:playlist}
\end{table}

\section*{Acknowledgements}
\label{sec:ack}
\noindent We would like to thank Mark Helman, Dominique Laurain, Arseniy Akopyan, Peter Moses, and Boris Odehnal for valuable insights. We are also indebted to the meticulous review by the referee. The first author is fellow of CNPq and coordinator of Project PRONEX/ CNPq/ FAPEG 2017 10 26 7000 508.

\appendix

\section{Explicit 3-Periodic Vertices}
\label{app:explicit}
\subsection{Pair 0: Confocal}
\label{app:explicit-confocal}

Let $(a,b)$ be the semi-axes of the external ellipse. Let $P_i=(x_i,y_i)/q_i$, $i=1,2,3$, denote the 3-periodic vertices, given by \cite{garcia2019-incenter}:

\begin{align*}
q_1&=1,\\
%\label{eqn:p2}
x_{2}&=-{b}^{4} \left(  \left(   a^2+{b}^{2}\right)k_1 -{a}^{2}  \right) x_1^{3}-2\,{a}^{4}{b}^{2} k_2  x_1^{2}{y_1}\\
&+{a}^{4} \left(  ({a
}^{2}-3\, {b}^{2})k_1+{b}^{2}
 \right) {x_1}\,y_1^{2}-2{a}^{6} k_2 y_1^{3},\\
y_{2}&= 2{b}^{6} k_2 x_1^{3}+{b}^{4}\left(  ({b
 }^{2}-3\, {a}^{2}) k_1  +{a}^{2}
  \right) x_1^{2}{y_1}\\
&+  2\,{a}^{2} {b}^{4}k_2 {x_1} y_1^{2} -{
a}^{4}  \left(  \left(   a^2+{b}^{2}\right)k_1  -{b}^{2}  \right)  y_1^{3},\\
q_2&={b}^{4} \left( a^2-c^2k_1   \right)
x_1^{2}+{a}^{4} \left(  {b}^{2}+c^2k_1  
 \right) y_1^{2} - 2\, {a}^{2}{b}^{2}{c^2}k_2 {x_1}\,{
y_1},\\
x_{3}&= {b}^{4} \left( {a}^{2}- \left( {b}^{2}+{a}^{2} \right) \right)
 k_1  x_1^{3} +2\,{a}^{4}{b}^{2}k_2  x_1^{2}{ y_1}\\
 &+{a}^{4} \left( 
  k_1 \left( {a}^{2}-3\,{b}^{2}
 \right) +{b}^{2} \right) { x_1}\, y_1^{2} +2\, {a}^{6} k_2 y_1^{3},\\
y_{3}&= -2\, {b}^{6} k_2 x_1^{3}+{b}^{4} \left( {a}^{2}+ \left( {b}^{2}-3\,{a}^{2} \right)    k_1 \right) {{ x_1}}^{2}{ y_1}
\\
& -2\,{a}^{2}  {b}^{4} k_2  x_1 y_1^{2}+
 {a}^{4} \left( {b}^{2}- \left( {b}^{2}+{a}^{2} \right)   k_1 \right)\,  y_1^{3},\\
q_3&= {b}^{4} \left( {a}^{2}-{c^2}k_1   \right) x_1^{2}+{a}^{4} \left( {b}^{2}+c^2k_1  \right)  y_1^{2}+2\,{a}^{2}{b}^{
2} c^2 k_2\, { x_1}\,{ y_1},
\end{align*}
\noindent where:

\begin{align*}
k_1&=\frac{d_1^2\delta_1^2}{\,d_2}=\cos^2{\alpha},\;\;k_2=\frac{ \delta_1d_1^2}{d_2 }\sqrt{ d_2 -d_1^4\delta_1^2}=\sin{\alpha}\cos\alpha,\\
c^2&=a^2-b^2,\;\; d_1=(a\,b/c)^2,\;\;d_2={b}^{4}x_1^2 +{a}^{4}y_1^2,\\
\delta&=\sqrt{a^4+b^4-a^2 b^2},\;\;\delta_1=\sqrt{2 \delta-a^2-b^2}\cdot
\end{align*}

\noindent where $\alpha$, though not used here, is the angle of segment $P_1 P_2$ (and $P_1 P_3$) with respect to the normal at $P_1$.

\subsection{Pair I: Incircle}
\label{app:explicit-I}

3-periodics are given by $P_1(t)=(x_1,y_1)=(a\cos{t},b\sin{t})$. Then, the $P_i=(x_i,y_i),i=2,3$ are:

\begin{align*}
	x_{2}=& 2\,{a}^{2}{b}^{2} \left( -{a}^{2}bx_1+ k \; y_1 \right)/q_2,\;\;y_{2}=-2\,a{b}^{3}  \left( {a}^{2}by_1+k\; x_1 \right)/q_2,\\
	x_{3}=&-2\,{a}^{2}{b}^{2} \left( {a}^{2}bx_1+k\;y_1 \right)/q_3,\;\;y_{3}=  2\,{b}^{3}a \left( -{a}^{2}by_1+k\; x_1 \right)/q_3,\\
	k=&\sqrt {{a}^{3} \left( a+2\,b \right)x_1^{2} +{a}^{2}b \left( 2\,a+b \right) y_1^{2}},\\
	q_2=&  2b^2(a+b)((a^2-b^2)x_1^2+a^2b^2),\\
	q_3=&\left( {b}^{2}{a}^{4}-y_1^{2}{a}^{4}+2\,{a}^{2}{b}^{4}+{a}^{2}
	{b}^{2}x_1^{2}-2\,x_1^{2}{b}^{4} \right)  \left(a+b\right) \cdot
\end{align*}

\subsection{Pair II: Inellipse}
\label{app:explicit-II}

3-periodics are given by $P_1(t)=(x_1,y_1)=R(\cos{t},\sin{t})$ with $R=a+b$. Then the	$P_i=(x_i,y_i),i=2,3$ are given by:

\begin{align*}
x_{2}=& \left(-{b}^{2}x_1+y_1\,s_x \right){k_x},\;\;y_{2}= -\left(y_1\,{a}^{2}+x_1\,s_y \right){k_y},\\
x_{3}=& -\left({b}^{2}x_1+y_1\,s_x \right){k_x},\;\;y_{3}= \left(-y_1\,{a}^{2}+x_1\,s_y \right){k_y},\\
s_x=&\sqrt {{a}^{3}(a+2 b)-(a^{2}-b^2)x_1^{2} },\;\;s_y=\sqrt {(a^{2}-b^2)y_1^{2}+ {b}^{3}(2a+b) },\\
k_x=&\frac{a}{\left(-a+b\right) x_1^{2}+{a}^{2} \left(a+b \right) },\;\;k_y=\frac{b}{\left(a-b\right) y_1^{2}+{b}^{2} \left( a+b \right) }\cdot
\end{align*}

\subsection{Pair III: Homothetic}
\label{app:explicit-III}

3-periodics are given by $P_1(t)=(x_1,y_1)=(a\cos{t},b\sin{t})$. Then $P_i=(x_i,y_i),i=2,3$ are:

\begin{align*}
  (x_2,y_2)=&  \left( \frac {\sqrt{3}\;ay_1-bx_1}{2b}, \frac {-\sqrt{3}\;bx_1-ay_1}{2a}\right),\\
(x_3,y_3)=&  \left(  \frac {-\sqrt{3}\;a{y_1}-b{x_1}}{2b}, \frac {\sqrt{3}\;b{x_1}-a{y_1}}{2a}\right)\cdot
\end{align*}

\section{Elliptic Loci}
\label{app:ell-axes}
Below we list triangle centers amongst $X_k$, $k=1,\ldots,200$ for each of the Poncelet pairs mentioned in this article, whose loci are either ellipses or circles.

\begin{itemize}
    \item 0. Confocal pair (stationary $X_9$)
    \begin{itemize}
	\item Ellipses: 1, 2, 3, 4, 5, 7, 8, 10, 11, 12, 20, 21, 35, 36, 40, 46, 55, 56, 57, 63, 65, 72, 78, 79, 80, 84, 88, 90, 100, 104, 119, 140, 142, 144, 145, 149, 153, 162, 165, 190, 191, 200. Note: the first 29 in the list were proved in \cite{garcia2020-ellipses}.
	\item Circles: n/a
	\end{itemize}
    \item I. Incircle: (stationary $X_1$)
    \begin{itemize}
    \item Ellipses: 2, 4, 7, 8, 9, 10, 20, 21, 63, 72, 78, 79, 84, 90, 100, 104, 140, 142, 144, 145, 149, 153, 191, 200.
    \item Circles: 3, 5, 11, 12, 35, 36, 40, 46, 55, 56, 57, 65, 80, 119, 165.
    \end{itemize}
    \item II. Inellipse (w/ circumcircle): (stationary $X_3$)
	\begin{itemize}
	\item Ellipses: 6, 49, 51, 52, 54, 64, 66, 67, 68, 69, 70, 113, 125, 141, 143, 146, 154, 155,  159, 161, 182, 184, 185, 193, 195.
	\item Circles: 2, 4, 5, 20, 22, 23, 24, 25, 26, 74, 98, 99, 100, 101, 102, 103, 104, 105, 106, 107, 108, 109, 110, 111, 112, 140, 156, 186.
	\end{itemize}
    \item III. Homothetic: (stationary $X_2$)
	\begin{itemize}
	\item Ellipses: 3, 4, 5, 6, 17, 20, 32, 39, 62, 69, 76, 83, 98, 99, 114, 115, 140, 141, 147, 148, 182, 187, 190, 193, 194.
	\item Circles: 13, 14, 15, 16.
	\end{itemize}
	%\item IV. Dual: (stationary: $X_4$)
	%\begin{itemize}
    %\item Ellipses: 2, 3, 5, 20, 64, 107, 122, 133, 140, 154.
    %\item Circles n/a
    %\end{itemize}
\end{itemize}

Semi-axes lengths for the elliptic loci of many triangle centers are available in \cite{garcia2021-ellipses-web}.

\section{Loci of Incenter and Symmedian in the Elliptic Billiard}
\label{app:loci-x1x6}
Over 3-periodics in the elliptic billiard, the locus of the incenter $X_1$ is an origin centered ellipse with axes $a_1$, $b_1$ given by \cite{garcia2019-incenter}:

\begin{equation*}
 a_1=\frac{\delta-{b}^{2}}{a},\;\;\;
 b_1=\frac{{a}^{2}-\delta}{b}\cdot
\end{equation*}

Over the same family, the locus of $X_6$ is a convex quartic given by \cite[Theorem 2]{garcia2020-ellipses}:

\begin{equation*}
  \text{locus $X_6$}:\;c_1 x^4+c_2 y^4+c_3 x^2 y^2+ c_4 x^2 + c_5 y^2 = 0,
\end{equation*}
\noindent where:
\begin{align*}
c_1=&b^4(5\delta^2-4(a^2-b^2)\delta -a^2 b^2),&c_2=&a^4(5\delta^2+4(a^2-b^2)\delta-a^2b^2), \\
c_3=&2a^2 b^2(a^2 b^2+3\delta^2),&c_4=&a^2 b^4(3 b^4+2(2 a^2-b^2)\delta-5\delta^2),\\
c_5=&a^4 b^2(3 a^4+2(2 b^2-a^2)\delta-5\delta^2),&\delta=&\sqrt{a^4-a^2 b^2+b^4}\cdot
\end{align*}

Note: this curve has an isolated point at the origin whose geometric meaning is not yet understood.

%\section{Isodynamic Pedals and Isogonic Antipedals}
%\label{app:equilaterals}
%\input{130_app_equilaterals}

\section{Table of Symbols}
\begin{table}[H]
\small
\begin{tabular}{|c|l|l|}
\hline
symbol & meaning & note \\
\hline
$O$ & center of concentric pair & \\
$a,b$ & ellipse semi-axes & \\
$s_i,s$ & sidelength and semiperimeter & $i=1,\ldots{N}$\\
$\theta_i$ & internal angle & \\
$L$ & perimeter & $\sum_i{s_i}$ \\
$L_2$ & sum of squared sidelengths & $\sum_i{s_i^2}$ \\
$K$ & Steiner's Curvature Centroid & \makecell[lt]{$\sum_i{w_i{P_i}}/\sum_i{w_i}$\\$w_i=\sin(2\theta_i)$}\\
\hline
$r,R$ & inradius, circumradius & \\
$d'$ & $|X_4-X_5|$ & \\
$r_h,R_h$ & inradius, circumradius of ortic & \\
$\omega$ & Brocard angle & $\tan(\omega)=4A/L_2$ \\
\hline
$X_1$ & incenter & \\
$X_2$ & barycenter & \\
$X_3$ & circumcenter & \\
$X_4$ & orthocenter & \\
$X_5$ & center of 9-point circle & \\
\hline
\end{tabular}
\caption{Symbols used.}
\label{tab:symbols}
\end{table}

\label{app:symbols}

\bibliographystyle{maa}
\bibliography{references,authors_rgk_v3} 

\end{document}